\documentclass{amsart}

\usepackage{amsfonts}
\usepackage{amssymb}
\usepackage{amscd}

\newtheorem{thm}{Theorem}[section]

\newtheorem{prop}[thm]{Proposition}

\theoremstyle{definition}
\newtheorem{dfn}[thm]{Definition}
\newtheorem{rmk}[thm]{Remark}
\newtheorem{ex}[thm]{Example}

\begin{document}

\title[Multiplier modules of Hilbert C*-modules revisited]{Multiplier modules of Hilbert C*-modules revisited}
\author{Michael Frank} 
\address{Hochschule f\"ur Technik, Wirtschaft und Kultur (HTWK) Leipzig, Fakult\"at Informatik und Medien, PF 301166, D-04251 Leipzig, Germany, ORCID 0000-0001-8972-2154.}
\email{michael.frank@htwk-leipzig.de, michael.frank.leipzig@gmx.de}



\date{May 2025}
\dedicatory{Dedicated to David Royal Larson}

\begin{abstract}
The theory of multiplier modules of Hilbert C*-modules is reconsidered to obtain more properties of these special Hilbert C*-modules. The property of a Hilbert C*-module to be a multiplier C*-module is shown to be an invariant with respect to the consideration as a left or right Hilbert C*-module in the sense of a imprimitivity bimodule in strong Morita equivalence theory. The interrelation of the C*-algebras of ''compact'' operators, the Banach algebras of bounded module operators and the Banach spaces of bounded module operators of a Hilbert C*-module to its C*-dual Banach C*-module, are characterized for pairs of Hilbert C*-modules and their respective multiplier modules. The structures on the latter are always isometrically embedded into the respective structures on the former. Examples are given for which continuation of these kinds of bounded module operators from the initial Hilbert C*-module to its multiplier module fails. However, existing continuations turn out to be always unique. Similarly, bounded modular functionals from both kinds of Hilbert C*-modules to their respective C*-algebras of coefficients are compared, and eventually existing continuations are shown to be unique.
\end{abstract}

\maketitle

%%%%%%%%%%%%%%%%%%%%%%%%%%%%%%%%%%%%

Multiplier modules of (full) Hilbert C*-modules appeared in the literature during investigations of extensions of Hilbert C*-modules in terms of short exact sequences. There are several equivalent approaches to the subject, e.g. the double centralizer type approach  \cite{EchRae_1995,Schweizer_2002,Daws_2010,Daws_2011,BKMS_2024,Delfin_2024} or a pure Hilbert C*-module approach \cite{BG_2004,BG_2003}, cf.~\cite{F_1997,Solel_2001} for the link between them. We follow the approach in \cite{BG_2004,BG_2003}. The goal was to generalize the extension theory of C*-algebras to the context of full Hilbert C*-modules. The notion of a multiplier module of a full Hilbert $A$-module over a C*-algebra $A$ is justified by \cite[Thm.~1.2]{BG_2004}: It is the largest essential extension of $X$ up to unitary modular isomorphism of Hilbert $M(A)$-modules, where $M(A)$ is the multiplier C*-algebra of $A$. For details see section 2. In the sequel to \cite{BG_2004,BG_2003}, large parts of the extension theory of Hilbert C*-modules have been described, e.g.~in \cite{B_2004,Brueckler_2004,AC_2011,Kolarec_2013,AB_2017,Jingming_2017,Guljas_2021}. Recently, J.~Taylor described local multiplier modules over local multiplier algebras in case of imprimitivity modules between (non-unital, in general) strongly Morita equivalent C*-algebras, cf.~\cite{Tay1,Tay2}.

The aim of the present note is to fill in some missing facts from the point of view of classical Hilbert C*-module theory and to describe pairs of Hilbert C*-modules and their multiplier modules from the point of view of their common and partially even differing properties. We get some results that contradict habits and opinions from Hilbert and Banach space theory. So, some new examples complement the existing points of view on the theory of Hilbert C*-modules.

Considering multiplier modules as full left Hilbert C*-modules $X$ over a C*-algebra $A$ we obtain that they are at the same time full right Hilbert C*-modules and multiplier modules over the respective C*-algebras ${\rm K}_A(X)$. This reminds special imprimitivity bimodules, however no new type of Morita equivalence can be derived. Considering pairs $(X,M(X))$ of Hilbert C*-modules and their multiplier modules we consider the interrelations of comparable types of operator algebras of bounded module operators over them and continuation problems of operators from $X$ to $M(X)$, the same for bounded modular functionals. So, for ''compact'' operator algebras ${\rm K}_A(X)$ is $*$-isometrically embedded into $K_{M(A)}(M(X))$, but the former might be smaller than the latter one. Furthermore, the Banach algebra of bounded module operators ${\rm End}_{M(A)}(M(X))$ is isometrically embedded into ${\rm End}_A(X)$, as well as the Banach space ${\rm End}_{M(A)}(M(X),M(X)')$ into ${\rm End}_A(X,X')$, and again the former might be smaller than the latter one. Consequently, not any operator of the smaller spaces on $X$ can be continued to a bounded operator on the larger spaces on $M(X)$ obeying strict convergence. If such a continuation exists it is unique. The same picture can be obtained for bounded modular functionals on $M(X)$ and on $X$, respectively, so we found examples of pairs of Hilbert C*-modules $(X,M(X))$ with $X^\bot = \{ 0 \}$ for which no general Hahn-Banach type theorem can be obtained. Remarkably, there does not exist any non-trivial bounded module map from $M(X)$ to $M(A)$ vanishing on $X$. 

%%%%%%%%%%%%%%%%%%%%%%%%%%%%%%%%%%%%
\section{Introduction}

We denote C*-algebras by $A,B$. In case a C*-algebra $A$ is non-unital, general C*-multiplier theory provides us with some derived structures like multiplier algebras $M(A)$, left and right multiplier algebras $LM(A)$ and $RM(A)$, resp., and quasi-multiplier spaces $QM(A)$. To calculate these linear spaces any faithful $*$-representation of $A$ on a Hilbert space $H$ can be used. The calculation environment is the von Neumann algebra generated by the faithfully $*$-represented C*-algebra $A$, or $B(H)$ itself, cf. \cite[Ch.~3]{Murphy}, \cite{Pedersen_1984}. The unital C*-algebra $M(A)$ is defined as 
\[
   M(A) = \{ m \in B(H): ma, am \in A \,\, {\rm for} \, {\rm any} \,\, a \in A \} \, .
\]
The Banach algebras $LM(A)=RM(A)^*$ can be defined as
\[
   LM(A) = RM(A)^* =  \{ m \in B(H): ma  \in A \,\, {\rm for} \, {\rm any} \,\, a \in A \} \, ,
\]
whereas the involutive Banach space $QM(A)$ can be obtained as
\[
   QM(A) = \{ m \in B(H): bma  \in A \,\, {\rm for} \, {\rm any} \,\, a,b \in A \} \, .
\]
Inventing different kinds of strict topologies intrinsic characterizations of these structures as certain topological completions of $A$ on bounded sets of $A$ are available. For comprehensive sources we refer to the book by  P.~Ara and M.~Matthieu \cite{AM_2003} and to \cite{Lin_1989,Lin_1991,BMSh_1994,F_1999_1}. Note, that multiplier algebras might admit an entire lattice of non-unital, two-sided, non-isomorphic ideals $A_\alpha$ such that $M(A_\alpha) = M(A_\beta)$ for any two of them, cf.~\cite{HW_2012}. Also, either $M(A)=LM(A)$ and $M(A)=QM(A)$ at the same time, or $LM(A)$ is strictly larger than $M(A)$ and $QM(A)$ is strictly larger than $LM(A)$, cf.~\cite[Cor.~4.18]{Brown}. 

We would like to consider Hilbert C*-modules over (non-unital, in general) C*-algebras and, in particular, their multiplier modules and related structures. By convention all Hilbert C*-modules are right C*-modules, at the first glance. However, for full Hilbert $A$-modules one can obtain an operator-valued inner product turning them into full left Hilbert ${\rm K}_A(X)$-modules, and vice versa. So, the point of view decides which of the two Hilbert C*-module structures is primary and which is secondary. This kind of consideration will be used frequently in the present paper.

A pre-Hilbert $A$-module over a C*-algebra $A$ is an $A$-module $X$ equipped with an $A$-valued, non-degenerate mapping $ \langle .,.
\rangle : X \times X \to A$ being conjugate-$A$-linear in the first argument and $A$-linear in the second one, and satisfying $\langle x,x \rangle \geq 0$ for every $x \in X$. The map $\langle .,. \rangle$ is called the $A$-valued inner product on $X$. A pre-Hilbert $A$-module $\{ X, \langle .,. \rangle \}$ is Hilbert if and only if it is complete with respect to the norm $\| . \| = \| \langle .,. \rangle \|^{1/2}_A$. We always assume that the complex linear structures of
$A$ and $X$ are compatible. A Hilbert $A$-module $\{ X, \langle .,. \rangle \}$ over a C*-algebra $A$ is full if the norm-closed $A$-linear hull $\langle X,X \rangle$ of the range of the inner product coincides with $A$. Two (full) Hilbert $A$-modules $\{ X , \langle .,. \rangle_X \}$ and $\{ Y, \langle .,. \rangle_Y \}$ over a fixed C*-algebra $A$ are unitarily equivalent (or unitarily isomorphic) iff there exists a bounded invertible adjointable map $T: X \to Y$ such that $\langle .,. \rangle_X = \langle T(.),T(.) \rangle_Y$ on $X$. The $A$-dual Banach $A$-module $X'$ of a Hilbert $A$-module $X$ is defined as the set of all bounded $A$-linear maps from $X$ into $A$. It might not be a Hilbert $A$-module itself, cf.~\cite{F_1999_1}. But, $X$ is always canonically isometrically embedded into $X'$ as a Banach $A$-submodule via the identification of $x \in X$ with $\langle x, \cdot \rangle \in X'$. 
Note, that two $A$-valued inner products on a Hilbert $A$-module $X$ inducing equivalent norms on $X$ might not be unitarily isomorphic, cf.~\cite{Brown,F_1999_1}. Thus, full Hilbert C*-modules are always a triple of the C*-algebra $A$ of coefficients, the Banach $A$-module $X$ and the $A$-valued inner product on $X$. We omit the explicit reference to the C*-valued inner product in situations where its definition formula is standard or only its existence is important. 

We are interested in properties of sets of bounded $A$-linear operators between Banach and Hilbert C*-modules $X$. The set ${\rm End}_A(X)$ of all bounded module operators on Hilbert $A$-modules $X$ forms a Banach algebra, whereas the set ${\rm End}_A^*(X)$ of all bounded module operators which possess an adjoint operator inside ${\rm End}_A(X)$ has the structure of a unital C*-algebra. Note, that these two sets do not coincide in general, cf.~\cite{Paschke,F_1999_1}. An important subset of ${\rm End}_A^*(X)$ is the set ${\rm K}_A(X)$ of ''compact'' operators, which is defined as the norm-closure of the set of all finite linear combinations of elementary operators
\[
  \{ \theta_{a,b} \in {\rm End}_A(X) : a,b \in X \, , \:
  \theta_{a,b}(c) = b \langle a,c \rangle \quad {\rm for} \: {\rm every}
  \quad c \in X \}.
\]
It is a C*-subalgebra and a two-sided ideal of ${\rm End}_A^*(X)$. In contrast to the well-known situation for Hilbert spaces, the properties
of an operator of being ''compact'' or possessing an adjoint depend strongly on the choice of the $A$-valued inner product on $X$, i.e.~these
properties are not invariant for unitarily non-isomorphic C*-valued inner products on $X$ inducing equivalent norms, cf.~\cite{F_1999_1}.

We postpone a detailled introduction to multiplier modules of Hilbert C*-modules to the next section. Our standard references to Hilbert C*-module theory are \cite{Lance_95,RW_1998,Wegge-Olsen}.

Searching for intrinsic characterizations of Hilbert C*-modules, especially over non-unital C*-algebras, the notions of orthonormal bases and of frames for Hilbert spaces were rediscovered in the modular context of Hilbert C*-modules by D.~R.~Larson and the author during 2018-2022, \cite{FL_1999, 
FL_2000,FL_2002} and cf.~\cite{HJLM}. The new theory started with the norm-convergent case. Remarkably, there was a shift in significance towards modular frames, since far not all Hilbert C*-modules admit orthogonal bases. Also, certain classes of Hilbert C*-modules do not possess modular frames, however the most usual classes of Hilbert C*-modules do. Theory and applications have been developed and extended since then. The type of convergence of the defining series in the modular context has been widened to strict, weak, weak* types or algebraic order type of convergences in cases. 

In 2017 Lj.~Aramba{\v{s}}i\'c and D.~Baki\'c have made  a significant progress in the understanding of Hilbert C*-modules over non-unital C*-algebras. They used the strict completion picture to multiplier modules of $X$ and introduced so called outer frames of the multiplier modules $M(X)$ to extend the available sets of norm-convergent or strictly convergent modular frames of the related initial Hilbert C*-modules $X$. It turned out that all outer and inner frames of countably generated or algebraically finitely generated Hilbert C*-modules $X$ in the sense of strict convergence can be characterized by surjections of either ${\rm End}_A^*(l_2(A),X)$ or ${\rm End}_A^*(A^N,X)$ for some $N\in {\mathbb N}$, resp., \cite[Thm.~3.18, Thm.~3.19, Prop.~3.22, Prop.~3.23]{AB_2017}. By the way, the notion of countably generatedness of Hilbert C*-modules over non-unital C*-algebras has been formulated more precise. Lateron, M.~Naroei Irani and A.~Nazari made steps towards inner and outer modular woven frames, cf.~\cite{NI_N_2021}. 

We are partially interested in strong Morita equivalence of C*-algebras and in imprimitivity bimodules realizing those equivalences. A C*-correspondence from a C*-algebra $B$ to a C*-algebra $A$ (or an $B$-$A$ C*-correspondence) is a Hilbert module $X$ over $A$ together with a nondegenerate $*$-homomorphism from $B$ to ${\rm End}^*_A(X)$, which introduces a left $B$-module structure on $X$. For any $B$-$A$ C*-correspondence $X$ the actions extend uniquely to multiplier algebras of $B$ and of $A$ so that we may always treat $X$ as a $M(B)$-$M(A)$ C*-correspondence. To relate it to the multiplier $M(B)$-$M(A)$ bimodule $M(X)$, let us note that whenever $Y$ is a Hilbert $A$-module, then ${\rm End}^*_A(Y,X)$
is naturally a C*-correspondence from $M(B)$ to ${\rm End}^*_A(Y)$. In particular, ${\rm End}^*_A(Y,X)$ contains ${\rm K}_A(Y,X)$ as a sub-C*-correspondence.
An $B$-$A$ imprimitivity bimodule is a full $B$-$A$ C*-correspondence $X$ which is also a full left Hilbert $B$-module and the two $B$- resp. $A$-valued inner products satisfy the equality $\langle x,y \rangle_B z = x \langle y,z \rangle_A$ for an $x,y,z \in X$. In particular, for imprimitivity bimodules $X$, $X$ is a full $B$-$A$ C*-correspondence such that the left action $\phi: B \to  {\rm End}^*_A(X)$ restricts to a $*$-isomorphism of $B$ with ${\rm K}_A(X)$. 
Two C*-algebras $B$ and $A$ are said to be strongly Morita equivalent iff there exists an imprimitivity bimodule (or a $B$-$A$ imprimitivity bimodule) connecting them. It is a real equivalence relation, so one can form classes of strongly Morita equivalent C*-algebras sharing important properties among their elements. The interested reader is referred to \cite{RW_1998,EchRae_1995,Daws_2010,Daws_2011,BKMS_2024,Delfin_2024} to learn more about the theory of strong Morita equivalence of C*-algebras and about C*-correspondences or imprimitivity bimodules. 
If $X$ and $Y$ are, respectively, $B$-$A$ and $C$-$B$ imprimitivity bimodules, then their algebraic tensor product $X \odot_B Y$ is a $C$-$A$ imprimitivity bimodule, and one can define two inner products on  $X \odot_B Y$ by
\[
\langle x \otimes z, y \otimes t \rangle_A = \langle \langle y,x\rangle_B z, t \rangle_A \, , \,\, 
\langle x \otimes z, y \otimes t \rangle_C = \langle x, y \langle t,z\rangle_B \rangle_C
\]
for any $x,y \in X$, $z,t \in Y$. This way  $X \odot_B Y$ becomes an $C$-$A$ imprimitivity bimodule. 
Remarkably, the set of all $B$-$A$ imprimitivity bimodules connecting two given strongly Morita equivalent C*-algebras $A$ and $B$ can be quite manifold. Using the defined tensor product operation the problem burns down to the consideration of the Picard group ${\rm Pic} (A)$ of one C*-algebra $A$ of unitary isomorphism classes of $A$-$A$ imprimitivity bimodules for an arbitrarily selected element $A$ of a given class of strongly Morita eqivalent to $A$ C*-algebras. Note, that the inverse of an imprimitivity bimodule $X$ in ${\rm Pic}(A)$ is its dual module ${\tilde X}$. If $Y$ is an $B$-$A$ imprimitivity bimodule, then the map $X \to {\tilde Y} \otimes_A X \otimes_A Y$ induces an isomorphism of ${\rm Pic}(A)$ and ${\rm Pic}(B)$, and \cite[Theorem 1.2]{BGR_1977} shows that the Picard groups are stable isomorphism invariant. For a non-trivial example cf.~\cite{Kodaka_1997} for the Picard groups of irrational rotation algebras. For a comprehensive account see \cite[sec. 6-7]{BW}.

%%%%%%%%%%%%%%%%%%%%%%%%%%%%%%%%%%%%
\section{On multiplier modules}

Let us show a way to define multiplier modules $M(X)$ of given (full) Hilbert $A$-modules $X$ as a certain related Hilbert $M(A)$-module, and let us provide some important properties of them. We start with a class of extensions of a given full Hilbert $A$-module $X$ as defined in \cite{BG_2004}:

\begin{dfn} (cf.~\cite[Def.~1.1]{BG_2004}) \label{dfn_extensions}
Let $X$ be a full Hilbert C*-module over a given (non-unital, in general) C*-algebra $A$. An extension of $X$ is a triple $(Y,B,\Phi)$ such that
\begin{itemize}
   \item[(i)$\,\,\,\,$] $B$ is a C*-algebra containing $A$ as a two-sided norm-closed ideal.
   \item[(ii)$\,\,$] $Y$ is a Hilbert $B$-module. 
   \item[(iii)] $\Phi: X \to Y$ is a bounded module map satisfying $\langle \Phi(x),\Phi(y) \rangle = \langle x,y \rangle$ for any $x,y \in X$.
   \item[(iv)] ${\rm Im}(\Phi) = YA = \{ za : z \in Y, a \in A \} = \{ x \in Y : \langle x,x \rangle \in A \}$ (by the Hewitt-Cohen factorization theorem, \cite[Thm.~4.1]{Pedersen_1998}, \cite[Prop.~2.31]{RW_1998}, \cite[Thm.~23.22]{Hewitt_Ross_1970}).
\end{itemize}
The triple $(Y,B,\Phi)$ is an essential extension of $X$ if $A$ is an essential ideal of $B$.
\end{dfn}

Note, that $\Phi$ is an $A$-linear isometry of Hilbert $A$-modules and, hence, in case of a surjective mapping a unitary map, preserving $A$-valued inner products up to unitary equivalence, cf.~\cite{Lance_1994}, \cite[Thm. 5]{F_1997}, \cite[Thm. 1.1]{Solel_2001}. So, $Y$ and $\Phi(Y)$ are unitarily equivalent Hilbert C*-modules. In the sequel we consider the C*-algebras $A$ and $B$ as a Hilbert $A$-module and as a Hilbert $B$-module over itself, respectively, setting $\langle a,b \rangle = a^*b$ for any two C*-algebra elements $a,b$. Then $A$ and $B$ are $*$-isometrically isomorphic to the C*-algebras ${\rm K}_A(A)$ and ${\rm K}_B(B)$, respectively. By \cite[Prop.~2.2.16]{Wegge-Olsen} these $*$-isomorphisms extend to $*$-isomorphisms of $M(A)$ and of $M(B)$ with the C*-algebras ${\rm End}_{M(A)}^*(A)$ and ${\rm End}_{M(B)}^*(B)$, respectively. We shall use these identifications freely. 

\begin{dfn} \label{dfn_multmod}
Let $X$ be a (not necessarily full) Hilbert C*-module over a given (non-unital, in general) C*-algebra $A$. Denote by $M(X)$ the set of all adjointable maps from $A$ to $X$, i.e.~$M(X) ={\rm End}_A^*(A,X)$. Obviously, $M(X)$ is a Hilbert $M(A)$-module with the $M(A)$-valued inner product $\langle z_1,z_2 \rangle = z_1^*z_2$ for $z_1, z_2 \in M(X)$. The resulting Hilbert $M(A)$-module norm coincides with the operator norm on $M(X)$. We call $M(X)$ the multiplier module of $X$. 
\end{dfn}

In \cite{BG_2004} $M(X)$ is shown to be an essential extension of $X$ identifying $X$ and ${\rm K}_A(A,X)$ as the subset $\{ za : z \in M(X), a \in A \}$ isometrically. In fact, $M(X)$ is the largest essential extension of $X$, because $M(A)$ is the largest essential extension of $A$ containing $*$-isomorphic copies of all C*-algebras $B$ which contain $*$-isomorphic copies of $A$ as an essential ideal, \cite[Thm.~1.2]{BG_2004}. This justifies the point of view on $M(X)$ as a Hilbert C*-module analog of the multiplier algebra in C*-theory. For a given Hilbert $A$-module $\{ X, \langle .,. \rangle_X \}$ the respective multiplier module $M(X)$ is unique up to unitary isomorphism of Hilbert $M(A)$-modules. 

Let us explain the definition of multiplier modules $M(X)$ for non-full Hilbert $A$-modules $X$. $\langle X,X \rangle$ is a two-sided norm-closed ideal of $A$. If $X$ is a non-full Hilbert $A$-module form the derived full Hilbert $A$-module $X_c$ by adding a copy of $A$ as an orthogonal direct summand to $X$, $X_c = X \oplus A$, $\langle .,. \rangle_c = \langle .,. \rangle_X+ \langle .,. \rangle_A$. Then construct its multiplier module $M(X_c)$. Using the mapping picture of Definition \ref{dfn_multmod} of $M(X_c)$ we see that $M(X_1 \oplus X_2) = M(X_1) \oplus M(X_2)$ using modular projection operators onto each of the orthogonal summands. So, we have an orthogonal decomposition $M(X_c) = M(X) \oplus M(A)$, taking the first orthogonal direct summand and Hilbert $M(A)$-module as the definition of the multiplier module $M(X)$ of $X$. Clearly, the (non-unital, in general) C*-algebra $\langle X,X \rangle$ is a two-sided norm-closed ideal of $A$, and the C*-algebra  $\langle M(X), M(X) \rangle$ is a two-sided norm-closed ideal of $M(A)$ containing $\langle X,X \rangle$. However, the latter might be non-unital, so we should be more careful with their C*-algebraic interrelations. In general,  $\langle X,X \rangle$ is a two-sided norm-closed ideal in $\langle M(X), M(X) \rangle$. The uniqueness results for the pairings $(X,M(X))$ are preserved.  It is important to realize, that $M(X)$ might depend on the choice of the C*-algebra acting on $X$, e.g. for $X=A$ with $A$ a non-unital C*-algebra we have $M_A(A)= {\rm End}_A^*(A,A)=M(A)$, but $M_{M(A)}(A)={\rm End}_{M(A)}^*(M(A),A)= A$. Last but not least, if the C*-algebra $A$ is non-unital then for $X=l_2(A)$ the multiplier module $M(l_2(A))$ equals to 
\[
   M(l_2(A))= \{ \{x_n\}_{n \in {\mathbb N}} : x_n \in M(A), \Sigma_n x_n^*x_n \,\ {\rm converges} \,\,{\rm strictly} \,\, {\rm w.r.t.} \,\, A \,\, {\rm in} \,\,M(A) \} \, ,
\]
cf.~\cite[Thm.~2.1]{BG_2004}. This gives a good non-trivial example, in particular, for certain non-$\sigma$-unital C*-algebras $A$, cf.~\cite[Ex.~2.2]{BG_2004}. In particular, $l_2(M(A))$ is smaller-equal than $M(l_2(A))$, generally speaking. For $A=K(l_2)$ and $X=l_2(A)$ one obtains $M(A)=B(l_2)$ and $M(X)= l_2(M(A))' \equiv  M(A) = B(l_2)$, the $M(A)$-dual Banach $M(A)$-module of $l_2(M(A))$, where
\[
   l_2(M(A))' = \left\{ m=\{ m_i\}_{i=1}^n : m_i \in M(A), \, \left\| \sum_{i=1}^n m_i^*m_i \right\| \leq K_m < \infty \,\, {\rm for} \, {\rm any} \,\, n \in {\mathbb N} \right\}
\]
 -- a self-dual Hilbert W*-module over $M(A)=B(l_2)$, because $A=K(l_2)$ is stable, cf.~\cite[Ex.~2.2]{BG_2004} and \cite{Paschke}.

\begin{prop} $\,$
\begin{enumerate}
  \item For a given pair of C*-algebras $(A,M(A))$, let $X_1,X_2$ be two full Hilbert C*-modules over $A$ such that their multiplier modules $M(X_1),M(X_2)$ are unitarily isomorphic as Hilbert $M(A)$-modules. Then $X_1$ and $X_2$ are unitarily isomorphic as Hilbert $A$-modules to $M(X_1)A \equiv M(X_2)A$. so, the pairings $(X, M(X))$ are bound to each other for given C*-algebras $(A,M(A))$ up to unitary equivalence.
  \item Suppose, we have two non-$*$-isomorphic C*-algebras $A_1$ and $A_2$ such that they admit the same multiplier C*-algebra $M(A)$. Let $X_1$ be a full Hilbert $A_1$-module and $X_2$ a full Hilbert $A_2$-module such that $M(X_1)$ and $M(X_2)$ are unitarily isomorphic as Hilbert $M(A)$-modules. Then $X_1$ is not unitarily isomorphic to $X_2$ as a Hilbert $M(A)$-module. 
\end{enumerate}
\end{prop}

The first assertion follows from (iv) of Definition \ref{dfn_extensions} combining  $\Phi_2$ with $\Phi_1^{-1}$ with respect to the set $\{ za : z \in M(X), a \in A \}$ in the largest essential extension $M(X)$, cf. \cite[Prop.~1.4]{RT2002}. The second assertion can be illustrated by example (see below), however $M(X_1)A_1$ is obviously not isometrically isomorphic to $M(X_2)A_2$ by supposition.

\begin{ex}
Let $H$ be an infinite-dimensional Hilbert space, and $K(H)$ and $B(H)$ the sets of compact linear operators and of bounded linear operators on it, respectively.  Consider the three C*-algebras
\[
   A_2 = \left(  \begin{array}{cc}  K(H) & 0 \\ 0 & K(H) \end{array} \right) , \,\,
   A_1 = \left(  \begin{array}{cc}  K(H) & 0 \\ 0 & B(H) \end{array} \right) , \,\,
\]
\[
   M(A_1) = M(A_2) = \left(  \begin{array}{cc}  B(H) & 0 \\ 0 & B(H) \end{array} \right) .
\]
Let us describe these C*-algebras as (extended) Hilbert C*-modules over $A_1$ and find their multiplier modules with respect to $A_1$.
\[
   X_2 =  \left(  \begin{array}{cc}  K(H) & 0 \\ 0 & K(H) \end{array} \right) \oplus \left(  \begin{array}{cc}  K(H) & 0 \\ 0 & B(H) \end{array} \right), \,\,
\]
\[
   X_1 =  \left(  \begin{array}{cc}  K(H) & 0 \\ 0 & B(H) \end{array} \right) \oplus \left(  \begin{array}{cc}  K(H) & 0 \\ 0 & B(H) \end{array} \right), \,\,
\]
\[
   M(X_2) = \left(  \begin{array}{cc}  B(H) & 0 \\ 0 & K(H) \end{array} \right) \oplus \left(  \begin{array}{cc}  B(H) & 0 \\ 0 & B(H) \end{array} \right) .
\]
\[
   M(X_1) = \left(  \begin{array}{cc}  B(H) & 0 \\ 0 & B(H) \end{array} \right) \oplus \left(  \begin{array}{cc}  B(H) & 0 \\ 0 & B(H) \end{array} \right) .
\]
By \cite[Prop.~1.1]{RT2002} both $X_1$ and $X_2$ admit a canonical isometric modular embedding into $M(X_1)$ and into $M(X_2)$, respectively, as Hilbert $M(A_1)$-submodules.  However, the multiplier module of $A_2$ with respect to $A_1$, $M_{A_1}(A_2)$, is a non-unital C*-algebra. 

Now, consider $X_1$ as a full Hilbert $A_1$-module and
\[
   X_3 =  \left(  \begin{array}{cc}  K(H) & 0 \\ 0 & K(H) \end{array} \right) \oplus \left(  \begin{array}{cc}  K(H) & 0 \\ 0 & K(H) \end{array} \right), \,\,
\]
as a full Hilbert $A_2$-module. Then $M(X_1)$ is unitarily isomorphic to $M(X_3)$ as a Hilbert $M(A_1) \equiv M(A_2)$-module, but $X_1$ and $X_3$ are not. 
\end{ex}

In general, for unital C*-algebras $A=M(A)$ any Hilbert $A$-module $X$ is its own multiplier module $M(X)=X$, \cite[Remark 1.11]{BG_2004}. To see that, an intrinsic topological characterization of multiplier modules for $\langle X,X \rangle \subseteq A$ being an essential ideal of $M(A)$  is useful. 

By \cite{BG_2004} there exists a suitable variant of a strict topology on multiplier modules $M(X)$: Let $A$ be a C*-algebra and $X$ be a Hilbert $A$-module.  Let the strict topology on $M(X)$ be induced jointly by the two families of semi-norms $\{\| z \to za \|: a \in A \}$ and $\{ \| \langle z,x \rangle \|: x \in X, \|x\| \leq 1 \}$ for $z \in M(X)$. It is a locally convex topology. The multiplier module $M(X)$ turns out to be complete with respect to this strict topology, and $M(X)$ is the strict completion of $X$, \cite[Thm.~1.8, 1.9]{BG_2004}. Moreover, the strict completion is an idempotent operation, i.e.~$M_{M(A)}(M_A(X))=M_A(X)$. Consequently, for unital C*-algebras $A=M(A)$ and Hilbert $A$-modules $X$ we have $X=M(X)$. The same is true whenever ${\rm K}_A(X)$ is unital, cf.~\cite[Cor.~2.9]{BG_2004}. 

For $X,Y$ Hilbert $A$-modules each operator $T \in {\rm End}_A^*(X,Y)$ has an extension $T_M \in {\rm End}_{M(A)}^*(M(X),M(Y))$ of the same norm value obtained as the strict continuation of $T$. Therefore, it is uniquely determined. Moreover, every operator in ${\rm End}_{M(A)}^*(M(X),M(Y))$ arises this way, i.e.~the C*-algebras ${\rm End}_{M(A)}^*(M(X),M(Y))$ and ${\rm End}_A^*(X,Y)$ are $*$-isomorphic, \cite[Thm.~2.3]{BG_2004}. Since a full Hilbert $A$-module $X$ is not only a right Hilbert $A$-module, but also a full left Hilbert ${\rm K}_A(X)$-module by the theory of strong Morita equivalence of the C*-algebras $A$ and ${\rm K}_A(X)$, $X$ can be considered as a C*-correspondence or as an imprimitivity bimodule of these two C*-algebras. By \cite{Kasparov_1980} the C*-algebra ${\rm End}_A^*(X)$ can be considered as the multiplier algebra of the C*-algebra ${\rm K}_A(X)$. Consequently, $M(X)$ is a full left Hilbert $M({\rm K}_A(X))$-module. So, for (full) Hilbert ${\rm K}_A(X)$-modules $X$ the Hilbert $M({\rm K}_A(X))$-module $M(X)$ can be identified with the (left) multiplier module of $X$ w.r.t. the pairing $({\rm K}_A(X), M({\rm K}_A(X)))$, too. This makes the property of a (full) Hilbert C*-module to be a multiplier module invariant under the choice of the point of view as a (full) left or right Hilbert C*-module. For similar thoughts compare with \cite[pp.~20-21, (a)-(b)]{BG_2004}, \cite[Props.~1.3, 1.6, Remark]{EchRae_1995}.

\begin{thm}
Let $A$ be a C*-algebra and $M(A)$ be its multiplier algebra. Let $X$ be a full (right) Hilbert $A$-module and $M(X)$ be its multiplier module, a full (right) Hilbert $M(A)$-module. Then $M(X)$ is also the full (left) multiplier module of the (left) Hilbert ${\rm K}_A(X)$-module $X$ with respect to the pairing of C*-algebras ${\rm K}_A(X)$ and $M({\rm K}_A(X))={\rm End}_A^*(X)=M({\rm K}_{M(A)}(M(X)))={\rm End}_{M(A)}^*(M(X))$, and vice versa. 
\end{thm}

\begin{proof}
By  \cite[Thm.~1.8, 1.9]{BG_2004} the unit ball of $M(X)$ is complete w.r.t. the locally convex strict (right) topology induced jointly by the two families of semi-norms $\{ \| z \to za \|: a \in A \}$ and $\{ \| \langle z,x \rangle_r \|: x \in X, \|x\| \leq 1 \}$ for $z \in M(X)$. Also, the unit ball of $X$ generates the unit ball of $M(X)$ strictly. By \cite[Def.~2, Remark 3]{B_2004} the operator strict topology on $X$ defined by the joint family of semi-norms $\{\| x \to T(x) \| :  T \in {\rm K}_A(X)\}$ and $\{ \|  x \to xa \|: a \in A \}$ with $x \in X$ coincides with the strict (right) topology on bounded sets of $X$ and of $M(X)$, and so on unit balls, in particular. Consequently, $X$ is dense in $M(X)$ w.r.t. the operator strict topology on $X$ and on $M(X)$. Note, that the operator strict topology is symmetric for imprimitivity bimodules $X$ in $A$ and ${\rm K}_A(X)$. 
By strong Morita equivalence via the imprimitivity bimodule $X$ we can symmetrically conclude, that the unit ball of $X$ is dense w.r.t. the locally convex strict (left) topology induced jointly by the two families of semi-norms $\{ \| z \to T(z) \|: T \in {\rm K}_A(X) \}$ and $\{ \| \langle z,x \rangle_l \|: x \in X, \|x\| \leq 1 \}$ for $z \in M(X)$. This gives the argument by \cite[Thm.~1.8, 1.9]{BG_2004} applied to the operator point of view on $X$ and on $M(X)$.
The $*$-isomorphisms of the respective operator C*-algebras are shown in \cite[Thm.~2.3]{BG_2004}.
\end{proof}

\begin{rmk}
Let us check the idea of a stronger kind of strong Morita equivalence restricting the set to multiplier module imprimitivity bimodules. There is a good class of pairwise strongly Morita equivalent C*-algebras $\{{\mathbb C}, M_n({\mathbb C}), K(H): n \in {\mathbb N}, H \,\, {\rm any} \, {\rm infinite{-}dimensional} \,\, {\rm Hilbert} \, {\rm space} \}$, because it has been well investigated in \cite{B_2004,BG_2004,Guljas_2021}. One has multiplier imprimitivity bimodules whenever the left and/or right Hilbert C*-module structure of imprimitivity bimodules involves  $\{{\mathbb C}, M_n({\mathbb C}):n \in {\mathbb N} \}$. The C*-algebras $\{ K(H):  H \,\, {\rm any} \, {\rm infinite{-}dimensional} \,\, {\rm Hilbert} \, {\rm space} \}$ may appear among the strongly Morita equivalent ones at the other end. What about imprimitivity bimodules $X$ connecting C*-algebras of type $K(H)$  for infinite-dimensional Hilbert spaces $H$? The minimal requirement to an equivalence relation is that an object has to be equivalent to itself. Take $K(H)$ and a full (left) Hilbert $K(H)$-module serving as a imprimitivity bimodule $X$ of $K(H)$ with itself, i.e. with the structure of a full (right) Hilbert $K(H)$-module. Then the full multiplier C*-module $M(X)=B(H)$ of $X$ is a imprimitivity bimodule of $B(H)$ with itself and does not belong to the set of $K(H)$-$K(H)$ imprimitivity bimodules any more. So, the set of $K(H)$-$K(H)$ imprimitivity bimodules does not contain any multiplier module, similarly for Hilbert spaces $H$ of pairwise non-isomorphic infinite-dimensional dimensions. The concept does not work.
\end{rmk}

%%%%%

\section{On modular operators and functionals}

The aim of this section is the investigation of the Banach algebras of all bounded module maps ${\rm End}_A(X)$ and ${\rm End}_{M(A)}(M(X))$ and their interrelations, as well as the sets of all bounded module maps $X'$ and $M(X)'$ over pairs of (full) Hilbert $A$-modules $X$ and their multiplier modules $M(X)$ over $M(A)$. To get non-trivial examples we need examples of C*-algebras $A$ such that their multiplier algebras $M(A)$ are strictly smaller than their left/right multiplier algebras. By \cite[Cor.~4.18]{Brown} either $M(A)=LM(A)$ and $M(A)=QM(A)$ at the same time, or $LM(A)$ is strictly larger than $M(A)$ and $QM(A)$ is strictly larger than $LM(A)$. For theory and examples see the existing literature on different types of multiplier algebras  and local multiplier algebras, e.g.~\cite{BMSh_1994}. We give a simple example following \cite[pp. 165-166]{Lin_1989}. 

Recall the C*-algebras $c_0$, $c$ and $l_\infty$ of all complex-valued sequences converging to zero, converging at all and being bounded in norm, respectively. Change the target C*-algebra $\mathbb C$ to the C*-algebra of all two-by-two valued matrices $M_2({\mathbb C})$. Consider the C*-algebra of all $M_2({\mathbb C})$-valued sequences with the sequence in the upper left corner converging at all and with the sequences derived from the other three positions being sequences converging to zero. We write $A$ as a symbol 
\[
 A = \left(   
\begin{array}{cc}
  c & c_0 \\ c_0 & c_0  
\end{array}
\right) \, .
\]

Then we can find the derived (left/right/two-sided) multiplier algebras / spaces:
\[
 M(A) = \left(   
\begin{array}{cc}
  c & c_0 \\ c_0 & l_\infty 
\end{array}
\right) \, ,  \,\, 
 LM(A) = RM(A)^* = \left(   
\begin{array}{cc}
  c & l_\infty \\ c_0 & l_\infty 
\end{array} 
\right) \, ,  
\]
\[
 QM(A) = \left(   
\begin{array}{cc}
  c & l_\infty \\ l_\infty & l_\infty 
\end{array} 
\right) \, .
\]
This can be calculated in the W*-algebra of all bounded $M_2({\mathbb C})$-valued sequences. Note, that $QM(A) = LM(A) + RM(A) \not= RM(A) \circ LM(A)$ for this particular example, a non-general situation. 

\begin{thm} 
Let $A$ be a C*-algebra with multiplier algebra $M(A)$. Let $X$ be a full Hilbert $A$-module and $M(X)$ be its full multiplier module.
\begin{enumerate}
  \item The C*-algebra ${\rm K}_A(X)$ of all ''compact'' operators on $X$ admits a $*$-isomor\-phic embedding into the C*-algebra ${\rm K}_{M(A)}(M(X))$ of all ''compact'' operators on $M(X)$. ${\rm K}_A(X)$ is smaller than ${\rm K}_{M(A)}(M(X))$ if $X \neq M(X)$. Nevertheless, their multiplier algebras are $*$-isomorphic, i.e. ${\rm End}_A(X)^* \cong {\rm End}_{M(A)}(M(X))^*$. If $X \neq M(X)$ then the embedding is not a surjection. 
  \item There does not exist any bounded $M(A)$-linear map $T_0: M(X) \to M(X)$ such that $T_0 \neq 0$ on $M(X)$, but $T_0 =0$ on $X \subseteq M(X)$.
  \item The Banach algebra ${\rm End}_{M(A)}(M(X))$ admits an isometric embedding into the Banach algebra ${\rm End}_A(X)$ by restricting an element on the domain from $M(X)$ to $X\subseteq M(X)$. If the left multiplier algebra of ${\rm K}_A(X)$ is larger than the multiplier algebra of it, then ${\rm End}_{M(A)}(M(X))$ can be smaller than ${\rm End}_A(X)$, i.e. not every bounded module operator on $X$ might admit an bounded module operator continuation on $M(X)$.  
\end{enumerate}
\end{thm}

\begin{proof}
Since $X \subseteq M(X)$ elementary ''compact'' operators on $X$ can be extended to $M(X)$ preserving their operator norm by the strict density of $X$ in $M(X)$. The C*-algebras of ''compact'' operators on Hilbert C*-modules are generated linearly by elementary operators w.r.t.~the operator norm. So the isometric $*$-isomorphic embedding of the C*-algebra ${\rm K}_A(X)$ into the C*-algebra ${\rm K}_{M(A)}(M(X))$ follows. However, the multiplier C*-algebras $M({\rm K}_A(X)) = {\rm End}_A(X)^*$ and $M({\rm K}_{M(A)}(M(X))) = {\rm End}_{M(A)}(M(X))^*$ are always $*$-isomorphic by \cite[Thm.~2.3]{BG_2004}.

Suppose, there exists a bounded $M(A)$-linear operator $T_0$ on $M(X)$ such that $T_0=0$ on $X \subseteq M(X)$, but $T_0(m) \neq 0$ for some $m \in M(X)$. Let $\{ x_\alpha: \alpha \in I\}$ be a net of elements of $X$ converging strictly to $m$, i.e.~the nets $\{ x_\alpha a: \alpha \in I\}$ converge to $ma \in X$ in norm for any $a \in A$. Consider the set $\{ \langle n, T_0(ma) \rangle : a \in A, n \in M(X) \}$. All these values are equal to zero by supposition. Since  $\langle n, T_0(ma) \rangle =  \langle n, T_0(m) \rangle a = 0$ for any $a \in A$ and $A$ is an essential ideal of $M(A)$ we conclude $\langle n, T_0(m) \rangle = 0$ for any $n \in M(X)$, forcing $T_0(m) = 0$, a contradiction to our assumption.

Restricting a bounded $M(A)$-linear operator on $M(X)$ on the domain from $M(X)$ to $X\subseteq M(X)$ one obtains a bounded $A$-linear operator on $X$. The norm is preserved since $X$ is strictly dense in $M(X)$, there does not exist any non-trivial bounded module operator on $M(X)$ vanishing on $X \subseteq M(X)$ and the norm is preserved on the subalgebra ${\rm End}_{(M(A)}(M(X))^*\equiv {\rm End}_A(X)^*$ by \cite[Thm.~2.3]{BG_2004}.  So, the restriction of a non-adjointable operator is a non-adjointable operator, again. In case, the left multiplier algebra $LM({\rm K}_A(X))= {\rm End}_A(X)$ is larger than the multiplier algebra $M({\rm K}_A(X))= {\rm End}_A(X)^*$ (see example above) not all elements of $LM({\rm K}_A(X)) \setminus M({\rm K}_A(X))$ are extendable to bounded module operators on $M(X)$:  Indeed, if $A = X$ is a C*-algebra with $LM(A) \supset M(A)$ then $M(X) = M(A)$ and all elements of $LM(A) \setminus M(A)$ cannot be continued from $X \subset M(X)$ to $M(X)$, cf. \cite[Thm.~1.5, 1.6]{Lin_1991}, \cite[Cor.~4.18]{Brown}.
\end{proof}

We found that not any bounded module operator on a Hilbert C*-module might admit a continuation to a bounded module operator of the same operator norm value on its multiplier module. However, if such a continuation with the same operator norm value exists it is unique.

\begin{rmk} \label{non-extend}
We demonstrate by example that the choice of the $A$-valued inner product on Hilbert $A$-modules $X$ within the class of $A$-valued inner products on $X$ inducing equivalent norms on $X$ may lead to other unitarily non-equivalent multiplier $M(A)$-modules. Return to the example at the beginning of the present section. If the C*-algebra $A$ defined there is equipped with the standard $A$-valued inner product as a Hilbert $A$-module $X$ then $M(X)=M(A)$. Now, modify this $A$-valued  inner product setting $\langle .,. \rangle_1 := \langle T(.),T(.) \rangle_A$ for 
\[
    T :=  \left( \begin{array}{cc} {1} & {1}\\ {0} & {1} \end{array} \right) \in LM(A) \setminus M(A) \, ,
\]
where $0$ is the zero sequence and $1$ is the identity sequence.
Clearly, $T$ is a non-adjointable invertible bounded $A$-linear operator on $A$, and $\langle .,. \rangle_1$ is an $A$-valued inner product on $A$ inducing an equivalent Hilbert module norm on $A = X$. A simple calculation for elements of $M(A)$ inside the C*-valued inner product $\langle .,. \rangle_1$ yields
\[
   \left\langle \left( \begin{array}{cc} {1} & {1}\\ {0} & {1} \end{array} \right) \circ
                   \left( \begin{array}{cc} c & c_0 \\ c_0 & l_\infty \end{array} \right) , 
                   \left( \begin{array}{cc} {1} & {1} \\ {0} & {1} \end{array} \right) \circ
                   \left( \begin{array}{cc} c & c_0 \\ c_0 & l_\infty \end{array} \right) 
   \right\rangle_A
   =
   \left( \begin{array}{cc} c & c \\ c & l_\infty \end{array} \right)  \not\in M(A) \, .
 \]
Consequently, the $A$-valued inner product $\langle .,. \rangle_1$ cannot be extended to $M(A)$, and $M(A)$ is not the multiplier module of the Hilbert $A$-module  $X_1=\{ A, \langle .,. \rangle_1\}$. So, the Banach $A$-module $A$ of the concrete example does not determine its multiplier module alone, one has to take into account the particular $A$-valued inner product on it.  
To calculate the multiplier module of $X_1$ one can use the von Neumann algebra of all  bounded $M_2({\mathbb C})$-valued sequences as an environment. Then
\begin{eqnarray*}
   \langle S(x),a \rangle_A & = & \langle x, S^*(a) \rangle_1 \\
                                      & = & \langle T(x), T(S^*(a)) \rangle_X \\
                                      & = & \langle T(x), (T^{-1})^*T^* (T(S^*(a)) \rangle_X \\
                                      & = & \langle T^{-1}T(x), T^* (T(S^*(a)) \rangle_X \\
                                      & = & \langle x, (T^*TS^*)(a) \rangle_X 
\end{eqnarray*}
for any $S \in {\rm End}_{A,1}^*(X,A)$, any $x \in X$, $a \in A$. Thus, $T^*TS^* \in {\rm End}_{A}^*(A,X) = \{ M(A), \langle .,. \rangle_{M(A)} \}$. 
\end{rmk}

In 2022 J.~Kaad and M.~Skeide published an example of a singular extension of the zero bounded C*-linear functional on a Hilbert C*-submodule $Y$ in a Hilbert C*-module $X$ where the orthogonal complement of $Y$ in $X$ was supposed to be the zero element of $X$, cf.~\cite{KS}. The author proved that such phenomena cannot appear for Hilbert C*-modules over monotone complete C*-algebras and for maximal one-sided ideals of C*-algebras, cf.~\cite{F_2024}. We should evaluate the pairs $(X,M(X))$ under consideration.

\begin{thm} \label{funct}
Let $A$ be a C*-algebra with multiplier algebra $M(A)$. Let $X$ be a full Hilbert $A$-module and $M(X)$ be its full multiplier module.
\begin{enumerate}
  \item There does not exist any bounded $M(A)$-linear map $f_0: M(X) \to M(A)$ such that $f_0 \neq 0$ on $M(X)$, but $f_0 =0$ on $X \subseteq M(X)$.
  \item The Banach $M(A)$-module $M(X)'_{M(A)}$ admits an isometric modular embedding into the Banach $A$-module $X'_A$ by restricting an element on the domain from $M(X)$ to $X\subseteq M(X)$. There exist examples such that $X'_A$ is strictly larger than the embedded copy of $M(X)'_{M(A)}$.
\end{enumerate}
\end{thm}

\begin{proof}
Suppose, there exists a bounded $M(A)$-linear functional $f_0 : M(X) \to M(A)$ such that $f_0=0$ on $X \subseteq M(X)$, but $f_0(m) \neq 0$ for some $m \in M(X)$. Let $\{ x_\alpha: \alpha \in I\}$ be a net of elements of $X$ converging strictly to $m$, i.e.~the nets $\{ x_\alpha a: \alpha \in I\}$ converge to $ma \in X$ in norm for any $a \in A$. Consider the set $\{ f_0(ma) : a \in A \}$. All these values are equal to zero by supposition. Since  $f_0(ma) = f_0(m) a = 0$ for any $a \in A$ and $A$ is an essential ideal of $M(A)$ we conclude $f_0(m) =0$, a contradiction to our assumption.

Restricting $f \in M(X)'$ to $X \subseteq M(X)$ we obtain a bounded $A$-linear functional of $X'$. The norm is preserved, since $X$ ist strictly dense in $M(X)$ and there does not exist any non-trivial bounded $M(A)$-linear functional on $M(X)$ vanishing on $X \subseteq M(X)$. The example in the beginning of the present section can be read as follows: Let $A$ be a C*-algebra such that $LMA) \supset M(A)$. Setting $A=X$ we get $X'=LM(A)$ and $M(X)'=M(A)'= M(A)$ since $M(X)=M(X)'$ is selfdual. So $X' \supset M(X)'$ and an $A$-valued bounded functional in $LM(A) \setminus M(A)$ cannot be continued from $X \subseteq M(X)$ to $M(X)$.
\end{proof}

Consequently, there does not exist any general Hahn-Banach type theorem for bounded C*-linear functionals for pairs $(X, M(X))$ of full Hilbert C*-modules $X$ and their multiplier modules $M(X)$, cf.~\cite{F_2002,F_2024}. However, if a bounded $A$-linear functional from a full Hilbert $A$-module to the C*-algebra $A$ admits a continuation to a bounded $M(A)$-linear functional from its multiplier module to $M(A)$ of the same norm value, then it is unique.

\begin{thm} 
Let $A$ be a non-unital C*-algebra with multiplier algebra $M(A)$. Let $X$ be a full Hilbert $A$-module and $M(X)$ be its full multiplier module.
\begin{enumerate}
  \item There does not exist any bounded $M(A)$-linear map $T_0: M(X) \to M(X)'$ such that $T_0 \neq 0$ on $M(X)$, but $T_0 =0$ on $X \subseteq M(X)$.
  \item The Banach space ${\rm End}_{M(A)}(M(X),M(X)')$ admits an isometric embedding into the Banach space ${\rm End}_A(X,X')$ by restricting an element on the domain from $M(X)$ to $X\subseteq M(X)$. There exist examples such that ${\rm End}_A(X,X')$  is strictly larger than the embedded copy of 
\linebreak[4]
${\rm End}_{M(A)}(M(X),M(X)')$.
\end{enumerate}
\end{thm}

\begin{proof}
Assume, there exists a bounded $M(A)$-linear map $T_0: M(X) \to M(X)'$ such that  $T_0 \neq 0$ on $M(X)$, but $T_0 =0$ on $X \subseteq M(X)$. Then there exists a non-zero element $m \in M(X)$ such that $T_0(m) \in M(X)'$,  $T_0(m) \neq 0$ on $M(X)$, but $T_0(m) = 0$ on $M(X)A = X$. This was excluded by Theorem \ref{funct}, (i), a contradiction. 

If we restrict an element $T \in {\rm End}_{M(A)}(M(X),M(X)')$ to $M(X)A = X$ we obtain an element $T \in {\rm End}_A(X,X')$ of the same operator norm, since $X$ is strictly dense in $M(X)$ and $M(X)'A = X'$. The algebraic operations are preserved. 

The example in the beginning of the present section shows: Let $A$ be a C*-algebra such that $LMA) \supset M(A)$. Setting $A=X$ we get ${\rm End}_A(X,X') = QM({\rm K}_A(X))$ and ${\rm End}_{M(A)}(M(X),M(X)') = QM({\rm K}_{M(A)}(M(X))$, cf.~\cite[Thm.~1.6]{Lin_1991}.  Since $QM({\rm K}_A(X)) \not= LM({\rm K}_A(X))$, but $QM({\rm K}_{M(A)}(M(X))) = M({\rm K}_{M(A)}(M(X))) = M(M(A)) = M(A)$ by \cite[Cor.~4.18]{Brown}, the assertion is demonstrated.
\end{proof}

Continuing Remark \ref{non-extend} the quasi-multiplier $T^*T$ of $A$ of the example induces a bounded modular map from $X=A$ to $X'=LM(A)$ that cannot be extended to a bounded modular map of the same norm from $M(X)=M(A)$ to $M(X)'=M(A)$.

\begin{rmk}
One natural question is whether multiplier modules might be C*-reflexive with respect to certain C*-algebras $B$ with $A \subseteq B \subseteq M(A)$, or not. The background are Definition \ref{dfn_extensions}, and a result by W.~L.~Paschke that the C*-valued inner product on a full Hilbert C*-module can be always continued to its C*-reflexive C*-bidual Banach C*-module preserving the isometric embedding of the former into the latter, cf.~\cite{Paschke_1974}. One could hope to find intrinsic topological or order-algebraical characterizations of some C*-reflexive Hilbert C*-modules, cf.~the $l_2(K(l_2))$-example above in the context of \cite{Paschke}. In general, C*-reflexivity of Hilbert C*-modules highly depends on both the inner structures of the C*-algebra of coefficients and of the Hilbert C*-module under consideration.

To test this point of view select a non-unital C*-algebra $A$ with $M(A)=LM(A)$. First, if one considers $A$-reflexivity of $X=A$ with the standard $A$-valued inner product then $M_A(X) = M(A)$, $X_A'=M(A)$ and $X_A'' = A$, so $M_A(X)$ is not $A$-reflexive. Furthermore, considering $X=A$ as a Hilbert $M(A)$-module, $M_{M(A)}(X)= A$, $X_{M(A)}'=M(A)$ and $X_{M(A)}'' = M(A)$, so it is not $M(A)$-reflexive, either. The deeper reason could be that condition (iv) of Definition \ref{dfn_extensions} might be violated for the C*-reflexive Hilbert C*-module extension of $X$.
\end{rmk}

\medskip 
Furthermore, instead of two-sided norm-closed ideals $A$ in their multiplier algebras $M(A)$ (and hence, multiplier modules) one-sided norm-closed ideals in C*-algebras and their biorthogonal complements in the host C*-algebra can be considered. Here several notions of density of the respective ideal in the arising C*-subalgebra of the host C*-algebra have to be compared, and it is far from obvious that they coincide like in the two-sided norm-closed ideal case, cf.~\cite{Manuilov_2025}. These considerations provide a wider variety of examples of Hilbert C*-modules and their multiplier modules.

%%%%%

\section*{Acknowledgement}
% \addcontentsline{toc}{section}{Acknowledgement}
I am grateful to David R.~Larson for the years of fruitful collaboration during 1998-2002.
David R.~Larson from Texas A{\&}M University made outstanding contributions to various mathematics subjects including operator theory, operator algebras, applied harmonic analysis, in particular wavelet and frame theory. As well he was engaged in organizing the highly impacted annual Great Plains Operator Theory Symposium (GPOTS) during many years.

I thank for a notice by Bartosz Kwa\'sniewski who attracted my attention to the double centralizer type approach to multiplier modules using only Banach C*-module constructions, and to problems about the equivalence of both these approaches in the case of Hilbert C*-modules, cf.~\cite{Daws_2010,Daws_2011,BKMS_2024,F_1997,Solel_2001}. 

%%%%%%


\begin{thebibliography}{1}

\bibitem{AC_2011}
{\sc  M.~Amyari, M.~Chakoshi}.
\newblock {P}ullback diagram of {H}ilbert {C}*-modules.
\newblock{\em Math. Commun.} {\bf 16}(2011), 569--575.

\bibitem{AM_2003}
{\sc P.~Ara, M.~Mathieu}.
\newblock {L}ocal {M}ultipliers of C*-{A}lgebras.
\newblock Springer Monographs in Mathematics, Springer-Verlag, London 2003/2012.

\bibitem{AB_2017}
{\sc Lj.~Aramba{\v{s}}i\'c, D.~Baki\'c}.
\newblock {F}rames and outer frames for {H}ilbert {C}*-modules.
\newblock {\em Lin. Multilin. Algebra} {\bf 65}(2017), issue 2, 381--431.

\bibitem{B_2004}
{\sc D.~Baki\'c}.
\newblock {A} class of strictly complete {H}ilbert {C}*-modules. 
\newblock manuscript, {8~pp.}, https://www.researchgate.net/profile/Damir-Bakic/publication \linebreak[4]
\newblock /242688089\_A\_class\_of\_strictly\_complete\_Hilbert\_C\_-modules/links/540ed1eb0cf2df04e7578e36 
\newblock /A-class-of-strictly-complete-Hilbert-C-modules.pdf.

\bibitem{BG_2004}
{\sc D.~Baki\'c, B.~Gulja{\v{s}}}. 
\newblock {E}xtensions of {H}ilbert {C}*-modules.
\newblock {\em Houston J. Math.} {\bf 30}(2004), 537--558.

\bibitem{BG_2003}
{\sc D.~Baki\'c, B.~Gulja{\v{s}}}. 
\newblock {E}xtensions of {H}ilbert {C}*-modules II.
\newblock {\em Glasnik Matemat{\v{c}}ki} {\bf 38(58)}(2003), 343--359.

\bibitem{BGR_1977}
{\sc L.~G.~Brown, P.~Green and M.~A.~Rieffel}.
\newblock {S}table isomorphism and strong {M}orita equivalence of {C}*-algebras.
\newblock {\em Pacific J. Math.} {\bf 71}(1977), 349--363.

\bibitem{BMSh_1994}
{\sc L.~G.~Brown, J.~Mingo, Nien-Tsu Shen}.
\newblock {Q}uasi-multipliers and embeddings of {H}ilbert {C}*-bimodules.
\newblock {\em Canad. J. Math.} {\bf 46}(1994), no.~6, 1150--1174.

\bibitem{Brown}
{\sc L.~G.~Brown}.
\newblock {C}lose hereditary {C}*-subalgebras and the strucure of quasi-multipliers.
\newblock MSRI preprint 1985, arxiv: 1501.07613v2, {\em Proc. A Royal Society of Edinburgh} {\bf 147}(2017), issue 2, 1--30.

\bibitem{Brueckler_2004}
{\sc F.~M.~Br\"uckler}.
\newblock {A} note on extensions of Hilbert C*-modules and their morphisms.
\newblock  {\em Glasnik Matemat{\v{c}}ki} {\bf 39(59)}(2004), 315--328.

\bibitem{BW}
{\sc H.~Bursztyn, St.~Waldmann}.
\newblock Completely positive inner products and strong Morita equivalence.
\newblock {\em Pacific J. Math.} {\bf 222} (2005), no. 2, 201--236.

\bibitem{BKMS_2024}
{\sc A.~Buss, B.~Kwa\'sniewski, A.~McKee, A.~Skalski}.
\newblock {F}ourier-{S}tieltjes category for twisted groupoid actions.
\newblock www.arXiv.org 2405.15653 (2024), 76 pp..

\bibitem{Daws_2010}
{\sc M.~Daws}.
\newblock {M}ultipliers, self-induced and dual Banach algebras.
\newblock  {\em Dissertationes Math.} {\bf 470(62)}(2010).

\bibitem{Daws_2011}
{\sc M.~Daws}.
\newblock {M}ultipliers of locally compact quantum groups via {H}ilbert {C}*-modules.
\newblock  {\em J. Lond. Math. Soc. (2)} {\bf 84(2)}(2011), 385--407.

\bibitem{Delfin_2024}
{\sc A.~Delf{\'{\i}}n (Ares de Parga)}.
\newblock {R}epresentations of *-correspondences on pairs of {H}ilbert spaces.
\newblock {\em J. Operator Theory} {\bf 92}(2024), issue 1, 167--188. / www.arXiv.org  2208.14605v4 (2024), 19pp., https://arxiv.org/pdf/2208.14605. 

\bibitem{EchRae_1995}
{\sc S.~Echterhoff, I.~Raeburn}.
\newblock {M}orita equivalence of crossed products.
\newblock  {\em Math. Scand.} {\bf 76}(1995), 289--309.

\bibitem{F_1997}
{\sc M.~Frank}.
\newblock {A} multiplier approach to the {L}ance-{B}lecher theorem.
\newblock{\em Zeitschr. Anal. Anwend.} {\bf 16}(1997), 565--573.

\bibitem{F_1999_1}
{\sc M.~Frank}.
\newblock {G}eometrical aspects of {H}ilbert {C}*-modules. 
\newblock {\em Positivity} {\bf 3}(1999), 215--243.

\bibitem{FL_1999}
{\sc M.~Frank, D.~R.~Larson}.
\newblock {A} module frame concept for {H}ilbert {C}*-modules,  in: The functional and harmonic analysis of wavelets and frames (San Antonio, TX, 1999).
\newblock {\em Contemp. Math.} {\bf 247}, 207–-233, Amer. Math. Soc., Providence, RI, 1999.

\bibitem{FL_2000}
{\sc M.~Frank, D.~R.~Larson}.
\newblock {M}odular frames for {H}ilbert {C}*-modules and symmetric approximation of frames.
\newblock  SPIE's 45th Annual Meeting, July 30 - August 4, 2000, San Diego, CA; Session 4119: Wavelet Applications in Signal and Image Processing VIII, org.: A.~Aldroubi, A.~F.~Laine, M.~A.~Unser, {\em Proceedings of SPIE} {\bf 4119}(2000), 325--336.

\bibitem{FL_2002}
{\sc M.~Frank, D.~R.~Larson}.
\newblock {F}rames in {H}ilbert {C}*-modules and {C}*-algebras.
\newblock {\em J. Operator Theory} {\bf 48}(2002), issue 2, 273--314.

\bibitem{F_2002}
{\sc M.~Frank}.
\newblock {O}n {H}ahn-{B}anach type theorems for {H}ilbert {C}*-modules.
\newblock {\em Internat. J. Math.} {\bf 13}(2002), 1--19.

\bibitem{F_2024}
{\sc M.~Frank}.
\newblock {R}egularity results for classes of {H}ilbert {C}*-modules with
respect to special bounded modular functionals.
\newblock{\em Annals Funct. Anal.} {\bf 15}(2024), article no.~19, 18 pp..

\bibitem{Guljas_2021}
{\sc B.~Gulja{\v{s}}}.
\newblock {H}ilbert {C}*-modules in which all relatively strictly closed submodules are complemented.
\newblock{\em Glasnik Matemati{\v{c}}ki} {\bf 56(76)}(2021), 343--374.

\bibitem{HJLM}
{\sc Deguang Han, Wu Jing, D.~R.~Larson, R.~N.~Mohapatra}.
\newblock {R}iesz bases and their dual modular frames in {H}ilbert {C}*-modules.
\newblock {\em J. Math. Anal. Appl.} {\bf 343}(2008), issue 1, 246--256.

\bibitem{Hewitt_Ross_1970}
{\sc E.~Hewitt, K.~A.~Ross}.
\newblock {A}bstract {H}armonic {A}nalysis. {V}ol.~II: {S}tructure and analysis of compact groups. {A}nalysis on compact abelian groups. 
\newblock {G}rundlehren {M}ath. {W}iss. 152, {S}pringer-{V}erlag, {B}erlin, 1970.

\bibitem{HW_2012}
{\sc T.~Hines, E.~Walsberg}.
\newblock {N}ontrivially {N}oetherian {C}*-algebras.
\newblock {\em Math. Scand.} {\bf 111}(2012), 135--146.

\bibitem{Jingming_2017}
{\sc Zhu Jingming}.
\newblock{G}eometric description of multiplier modules for {H}ilbert {C}*-modules in simple cases. 
\newblock {\em Ann. Funct. Anal.} {\bf 8}(2017), no. 1, 51--62.

\bibitem{KS} 
{\sc J.~Kaad, M.~Skeide}.
\newblock {K}ernels of {H}ilbert ${C}^{*}$-module maps.
\newblock {\em J. Oper. Theory}, {\bf 89}(2023), Issue 2, 343--348.

\bibitem{Kasparov_1980}
{\sc G.~G.~Kasparov}.
\newblock {H}ilbert {C}*-modules: {T}heorems of {S}tinespring and {V}oiculescu.
\newblock {\em J.~Oper.~Theory} {\bf 4}(1980), 133--150.

\bibitem{Kodaka_1997}
{\sc Kazunori Kodaka}.
\newblock {P}icard groups of irrational rotation algebras.
\newblock {\em J. London Math. Soc. (2)} {\bf 56}(1997), 179--188.

\bibitem{Kolarec_2013}
{\sc B.~Kolarec}.
\newblock {A} survey on extensions of {H}ilbert {C}*-modules.
\newblock{\em Quantum Probability and Related Topics} (2013), 209--221, DOI 10.1142/9789814447546{\_}0013.

\bibitem{Lance_1994}
{\sc E.~C.~Lance}.
\newblock {U}nitary operators on {H}ilbert {C}*-modules.
\newblock {\em Bull. London Math. Soc.} {\bf 26}(1994), 363--366.

\bibitem{Lance_95}
{\sc E.~C.~Lance}.
\newblock {H}ilbert {C}*-modules - a toolkit for operator algebraists.
\newblock {C}ambridge {U}niversity {P}ress, {C}ambridge, {E}ngland, {\em London Mathematical Society Lecture Note Series} {\bf 210}, 1995.

\bibitem{Lin_1989}
{\sc Huaxin Lin}.
\newblock The structure of quasi-mulitipliers of C*-algebras.
\newblock {\em Trans. Amer.Math. Soc.} {\bf 315}(1989), 147--172.

\bibitem{Lin_1991}
{\sc Huaxin Lin}.
\newblock {B}ounded module maps and pure completely positive maps.
\newblock {\em J.~Oper.~Theory} {\bf 26}(1991), 121--138.

\bibitem{Manuilov_2025}
{\sc V.~M.~Manuilov}.
\newblock {O}n large submodules in {H}ilbert {C}*-modules.
\newblock {\em J. Math. Anal. Appl.} {\bf 542}(2025), art.~no. 128781, 13pp..

\bibitem{Murphy} 
{\sc G.~J.~Murphy}.
\newblock ${C}^{*}$-{A}lgebras and {O}perator {T}heory.
\newblock Academic Press, Boston, 2004.

\bibitem{NI_N_2021}
{\sc M.~Naroei Irani, A.~Nazari}.
\newblock {T}he woven frame of multipliers in {H}ilbert {C}*-modules.
\newblock {\em Commun. Korean Math. Soc.} {\bf 36}(2021), issue 2, 257--266.

\bibitem{Paschke}
{\sc W.~L.~Paschke}.
\newblock {I}nner product modules over ${B}^{*}$-algebras.
\newblock {\em Trans. Amer. Math. Soc.} {\bf 182}(1973), 443--468.

\bibitem{Paschke_1974}
{\sc W.~L.~Paschke}.
\newblock {T}he double $B$-dual of an inner product module over a {C}*-algebra $B$.
\newblock {\em Can. J. Math.} {\bf XXVI}(1974), no.~5, 1272--1280.

\bibitem{Pedersen_1984}
{\sc G.~K.~Pedersen}.
\newblock {M}ultipliers of {AW}*-algebras.
\newblock {\em Math. Z.} {\bf 187}(1984), 23-24.

\bibitem{Pedersen_1998}
{\sc G.~K.~Pedersen}.
\newblock {F}actorizations in {C}*-algebras.
\newblock {\em Expos. Math.} {\bf 16}(1998), no.~2, 145--156.

\bibitem{RW_1998}
{\sc I.~Raeburn, D.~P.~Williams}.
\newblock {M}orita {E}quivalence and {C}ontinuous {T}race {C}*-algebras.
\newblock Mathematical Surveys and Monographs 60, Amer.~Math.~Soc., 1998.

\bibitem{RT2002}
{\sc I.~Raeburn, S.~J.~Thompson}. 
\newblock {C}ountably generated {H}ilbert modules, the {K}asparov stabilization theorem, and frames in {H}ilbert modules. 
\newblock {\em Proc.~Amer.~Math.~Soc.} {\bf 131}(2002),  1557--1564.

\bibitem{Schweizer_2002}
{\sc J.~Schweizer}.
\newblock {H}ilbert modules with a predual.
\newblock {\em J. Operator Theory} {\bf 48}(2002), 621--632.

\bibitem{Solel_2001}
{\sc B.~Solel}.
\newblock {I}sometries of {H}ilbert {C}*-modules.
\newblock {\em Trans. Amer. Math. Soc.} {\bf 353}(2001), no.~11, 4637--4660.

\bibitem{Tay1}
{\sc J.~Taylor}.
\newblock Aperiodic dynamical inclusions of C*-algebras.
\newblock Ph.D.~Thesis, Georg-August-Universit\"at G\"ottingen, G\"ottingen, Germany (2022),
http://dx.doi.org/10.53846/goediss-9727.

\bibitem{Tay2}
{\sc J.~Taylor}.
\newblock Aperiodic dynamical inclusions of C*-algebras.
\newblock www.arXiv.org  2303.10905v1 (2023), https://arxiv.org/pdf/2303.10905, 24pp..

\bibitem{Wegge-Olsen}
{\sc N.~E.~Wegge-Olsen}.
\newblock {K}-theory and {C}*-algebras - a friendly approach.
\newblock {O}xford {U}niversity {P}ress, {O}xford,  1993.

\end{thebibliography}
\end{document}